\documentclass[final,hyperfootnotes=false]{dmtcs-episciences}
\usepackage[utf8]{inputenc}
\usepackage{amsmath, amsfonts, amsthm, amssymb,scalerel, mathtools}
\usepackage{hyperref}
\usepackage[capitalize,noabbrev]{cleveref}
\allowdisplaybreaks

\newtheorem{theorem}{Theorem}[section]
\newtheorem{corollary}[theorem]{Corollary}
\newtheorem{lemma}[theorem]{Lemma}
\newtheorem{proposition}[theorem]{Proposition}
\theoremstyle{definition}

\numberwithin{equation}{section}

\newcommand{\N}{\mathbb{N}}
\DeclareMathOperator{\forb}{forb}
\DeclareMathOperator{\Avoid}{Avoid}

\newcommand{\bb}{\begin{bmatrix}}
\newcommand{\eb}{\end{bmatrix}}
\newcommand{\sbb}{\left[\begin{smallmatrix}}
\newcommand{\seb}{\end{smallmatrix}\right]}

\title{Exponential multivalued forbidden configurations}
\author{Travis Dillon\affiliationmark{1}\thanks{Research conducted under the auspices of the Budapest Semesters in Mathematics program.}
    \and Attila Sali\affiliationmark{2,3}\thanks{This author's work is partially supported by the National Research, Development and Innovation Office (NKFIH) [grant K--116769 and K--132696]; the National Research, Development and Innovation Fund [TUDFO/51757/2019-ITM, Thematic Excellence Program]; the BME NC TKP2020 grant of NKFIH Hungary; and the BME-Artificial Intelligence FIKP grant of EMMI [BME FIKP-MI/SC]. It is also connected to the ``Development of quality-oriented and harmonized R+D+I strategy and functional model at BME'' project supported by the New Hungary Development Plan [Project ID: TÁMOP-4.2.1/B-09/1/KMR-2010-0002].}
}
\affiliation{
    Lawrence University, WI, USA\\
    Alfr\'ed R\'enyi Institute of Mathematics, Budapest, Hungary\\
    Department of Computer Science, Budapest University of Technology and Economics, Budapest, Hungary
}
\received{2020-07-02}
\revised{2021-03-02}
\accepted{2021-03-08}

\begin{document}
\publicationdetails{23}{2021}{1}{7}{6613}
\maketitle
\begin{abstract}
\noindent
The \textit{forbidden number} $\operatorname{forb}(m,F)$, which denotes the maximum number of unique columns in an $m$-rowed $(0,1)$-matrix with no submatrix that is a row and column permutation of $F$, has been widely studied in extremal set theory. Recently, this function was extended to $r$-matrices, whose entries lie in $\{0,1,\dots,r-1\}$. The combinatorics of the generalized forbidden number is less well-studied. In this paper, we provide exact bounds for many $(0,1)$-matrices $F$, including all $2$-rowed matrices when $r > 3$. We also prove a stability result for the $2\times 2$ identity matrix. Along the way, we expose some interesting qualitative differences between the cases $r=2$, $r = 3$, and $r > 3$.
\end{abstract}

\section{Introduction}
We call a matrix \emph{simple} if it has no repeated columns. Every set system (or simple hypergraph) corresponds to a simple $(0,1)$-matrix via its element-set incidence matrix, and such matrices provide a convenient language for extremal set theory. We generalize this situation to $r$-\emph{matrices}, which have entries in $\{0,1,\ldots, r-1\}$. Such matrices can be thought of as $r$-coloured set systems or as representations of collections of functions from a given finite set into $\{0,1,\dots,r-1\}$.

For two  matrices $F$ and $A$, we say that $F$ is a \textit{configuration} of $A$, denoted $F \prec A$, if $A$ contains a submatrix which is a row and column permutation of $F$. If $F\nprec A$, we say that $A$ \textit{avoids} $F$. Configurations of simple $(0,1)$-matrices correspond to traces of set systems or hypergraphs. For a given finite collection $\mathcal F$ of matrices, we denote by $\Avoid(m,r,{\mathcal F})$ the collection of $m$-rowed, simple $r$-matrices that avoid every matrix $F \in \mathcal F$. We let $\lvert A \rvert$ denote the number of columns of $A$. The main extremal function in the study of forbidden configurations is
\[
    \forb(m,r,{\mathcal F})=\max\{|A| : A \in \Avoid(m,r,{\mathcal F})\}.
\]
When $r=2$ we usually write $\forb(m,{\mathcal F})$ in place of $\forb(m,2,{\mathcal F})$. We also use $\forb(m,r,F)$ instead of the more cumbersome $\forb(m,r,\{F\})$. We will use several simple properties of this function. For example, if $F \prec F'$ then $\forb(m,r,F) \leq \forb(m,r,F')$. Also, we let $F^c$ denote the \textit{complement} of the $(0,1)$-matrix $F$, where each 0 is replaced by a 1 and vice versa; then $\forb(m,r,F) = \forb(m,r,F^c)$.

The foundational result in the theory of forbidden configurations is Sauer's theorem (proven in \cite{Sauer}, also by Perles and Shelah \cite{Shelah} and Vapnik and Chervonenkis \cite{VC}). Let $K_k$ denote the complete $k\times 2^k$ simple $(0,1)$-configuration (corresponding to the power set of a $k$-element set).
\begin{theorem}\label{Sauer's Thm} For every positive integer $m$,
\[
\forb(m,K_k) = \binom{m}{k-1} + \binom{m}{k-2} + \cdots + \binom{m}{1} + \binom{m}{0}.
\]
\end{theorem}

Alon \cite{ALON} gave a generalization for complete $r$-matrices, but the forbidden number is exponential when $r > 2$. This is a special case of a more general phenomenon proved by F\"uredi and Sali \cite{FS}.

\begin{theorem}
Let ${\mathcal F}$ be a family of $r$-matrices. If for every pair $i,j\in\{0,1,\dots,r-1\}$ there is an $(i,j)$-matrix in ${\mathcal F}$, then $\forb(m,r,{\mathcal F}) = O(m^k)$ for some positive integer $k$. If ${\mathcal F}$ has no $(i,j)$-matrix for some pair $i,j \in \{0,1,\dots,r-1\}$, then $\forb(m,r,{\mathcal F}) = \Omega(2^m)$.
\end{theorem}

Extensive investigations have been undertaken for forbidden configurations of simple $(0,1)$-matrices; see, for example, the excellent dynamic survey of Anstee \cite{survey}. On the other hand, the more general case of $r$-matrices is not so well-explored. Previous papers mainly focus on providing bounds on the forbidden number for special classes of sets in the polynomial case \cite{ansteelu,ELLIS202024}. In this paper, we dive into exponential forbidden numbers and provide exact bounds when $(0,1)$-configurations of $r$-matrices are forbidden. This is similar in flavour to a recent paper of F\"uredi, Kostochka, and Luo \cite{furedi-exponential}, which proves several minimum-degree conditions that guarantee cycles in hypergraphs; by dropping the assumption of uniformity, their bounds jump from polynomial to exponential.

The structure of the paper is as follows. \cref{sec:general-bounds} provides a method to transfer bounds for $r=2$ to larger values of $r$.  The following three sections calculate forbidden numbers of specific classes of matrices. We obtain exact results when $r>3$ and bounds for $r=3$ that differ from the forbidden number by an additive constant. We also prove a stability result for the identity configuration. Our work culminates in \cref{sec:two-rows}, which provides exact forbidden numbers for all two-rowed $(0,1)$-configurations for every $r>3$ and a large class of two-rowed $(0,1)$ configurations when $r=3$. The main tool in both cases a reduction lemma. Finally, \cref{sec:three-rows} applies the method of \cref{sec:general-bounds} to obtain a nearly complete classification of $(0,1)$-configurations of size $3\times 2$ and $3 \times 3$.

\section{General bounds}\label{sec:general-bounds}

For a given configuration $A$, let $\bar{A}$ denote its underlying simple configuration. If $A$ has $m$ columns and $S \subseteq [m]$, then we let $A\vert_S$ be the restriction of $A$ to the rows with indices in $S$. By convention, we set $\forb(0,F) = 1$ for all $F$. In general, if $F$ has $t$ rows, then $\forb(k,F) = 2^k$ when $0 \leq k < t$.

\begin{lemma}\label{thm:upper-01-bound}
If $F$ is a $(0,1)$-matrix and $r \geq 3$, then 
\begin{equation}\label{eqn:upper-01-bound}
\forb(m,r,F) \leq \sum_{k=0}^m \binom{m}{k}(r-2)^{m-k}\forb(k,F).
\end{equation}
\end{lemma}
\begin{proof}
Let $A \in \Avoid(m,r,F)$, and let $X$ be a $k$-element subset of the rows. Consider the matrix $C$ obtained by taking all columns of $A$ that have 0's and 1's in exactly the rows in $X$, and let $C' = \overline{C\vert_X}$. We know that $\lvert C' \rvert \leq \forb(k,F)$. Each column in $C'$ appears with multiplicity at most $(r-2)^{m-k}$ in $C\vert_X$, so $\lvert C \rvert \leq (r-2)^{m-k}\forb(k,F)$. To finish the proof, we sum over all subsets of the rows.
\end{proof}

The bound given by this lemma may be quite bad, especially if $F$ is not simple. However, for simple matrices, we have the following lower bound.

\begin{lemma}\label{thm:lower-01-bound}
Let $F$ be a simple $(0,1)$-matrix with $n$ rows and fix $r \geq 3$. Suppose that $(A_k)_{k=1}^\infty$ is a sequence of $(0,1)$-matrices that avoids $F$, where $A_k$ has $k$ rows, such that $\overline{A_k\vert_S} \subseteq A_{n}$ for every $k \geq n$ and $S \in \binom{[k]}{n}$. If we set $\lvert A_0 \rvert = 1$, then
\begin{equation}
    \forb(m,r,F) \geq \sum_{k=0}^m \binom{m}{k}(r-2)^{m-k} \lvert A_k\rvert.
\end{equation}
\end{lemma}
\begin{proof}
We construct a configuration that avoids $F$ as follows. Let $k \in [m]$. For each $k$-set $X$ of rows, we choose the $(r-2)^{m-k}$ columns that contain a copy of $A_k$ in the rows of $X$ and have elements of $\{2,\dots,r-1\}$ in every other position. Let $A$ be the configuration that contains all such columns. If $F \prec A$, then $F \prec A\vert_{S}$ for some $n$-set of rows $S$. But every column in $A\vert_S$ appears in $A_n$, so $F \prec A_n$, a contradiction.
\end{proof}

The condition that $F$ is simple is absolutely essential. For simple matrices, however, this lemma can easily extend bounds from the classical case to the generalized one. In particular, combining \cref{thm:upper-01-bound,thm:lower-01-bound} proves the following.

\begin{lemma}\label{thm:01-equal}
Let $F$ be a simple $n$-rowed $(0,1)$-matrix. If there exists a sequence $(A_k)_{k=1}^\infty$ of $(0,1)$-matrices, each of which avoids $F$, such that\vspace{-1.3ex}
\begin{itemize}
    \setlength{\itemsep}{1pt}
    \setlength{\parskip}{0pt}
    \setlength{\parsep}{0pt}
    \item[\raisebox{1pt}{$\circ$}] $A_k$ has $k$ rows,\vspace{0.1em}
    \item[\raisebox{1pt}{$\circ$}] $\lvert A_k\rvert = \forb(k,F)$, and\vspace{0.2em}
    \item[\raisebox{1pt}{$\circ$}] $\overline{A_k|_S}$ is contained in $A_n$ for every $k\geq n$ and $n$-set $S \subseteq [k]$, then
\end{itemize}
\begin{equation}
    \forb(m,r,F) = \sum_{k=0}^m \binom{m}{k}(r-2)^{m-k}\forb(k,F).
\end{equation}
\end{lemma}

\section{Complete configurations}\label{sec:complete-configurations}

\begin{proposition}\label{thm:complete-config}
We have $\forb(m,r,K_k) = \sum_{i=0}^{k-1} \binom{m}{i} (r-1)^{m-i}$.
If $(r-1)^{m-k} \geq p-1$, then $\forb(m,r,p\cdot K_k) = \sum_{i=0}^{k-1}\binom{m}{i}(r-1)^{m-i} + (p-1)\binom{m}{k}$.
\end{proposition}
\begin{proof}
We first prove that $\forb(m,r,K_k) = \sum_{i=0}^{k-1}\binom{m}{i}(r-1)^{m-i}$. Let $A_n$ denote the $n$-rowed configuration that contains every column with at most $k-1$ zeros. Then $(A_n)$ satisfies the conditions of \cref{thm:01-equal}, so \cref{Sauer's Thm} implies that
\begin{align*}
    \forb(m,r,K_k)
    &= \sum_{n=0}^{m} \binom{m}{n} (r-2)^{m-n} \sum_{i=0}^{k-1} \binom{n}{i}\\
    &= \sum_{i=0}^{k-1} \binom{m}{i} \sum_{n=0}^m \binom{m-i}{n-i} (r-2)^{m-n}\\
    &= \sum_{i=0}^{k-1} \binom{m}{i}(r-1)^{m-i}.
\end{align*}

Now we prove the forbidden number for all $p$. The configuration that contains every column with at most $k-1$ zeros avoids $K_k$. If $(r-1)^{m-k} \geq p-1$, for each $k$-set of rows, we may append $p-1$ columns to this matrix that have zeros in that $k$-set and nowhere else. The resulting configuration avoids $p\cdot K_k$ and has $\sum_{i=0}^{k-1}\binom{m}{i}(r-1)^{m-i} + (p-1)\binom{m}{k}$ columns.

Now suppose that $A \in \Avoid(m,r,p\cdot K_k)$. For each $k$-set $X$ of rows, there is a column of $K_k$ that appears at most $p-1$ times in $A\vert_X$. Let $A'$ be the configuration obtained by deleting the corresponding columns of $A$ for all $k$-sets. Since $K_k$ is symmetric, no row-permutation of $K_k$ is a subset of $A'\vert_X$, so $K_k \nprec A'\vert_X$ for every $k$-set $X$. Therefore $K_k \nprec A'$, which implies that
\[
    \lvert A\rvert
    \leq \lvert A'\rvert + (p-1)\binom{m}{k}
    \leq \sum_{n=0}^{k-1}\binom{m}{n}(r-1)^{m-n} + (p-1)\binom{m}{k},
\]
as claimed.
\end{proof}

\cref{thm:complete-config} is enough to determine the logarithmic growth rate of $\forb(m,r,F)$ asymptotically for every $(0,1)$-configuration $F$.

\begin{corollary}\label{thm:asymptotic-log}
    The asymptotic formula $\log \forb(m,r,F) \sim m \log(r-1)$ holds as $m\to \infty$ for every fixed $(0,1)$-configuration $F$ and $r \geq 3$.
\end{corollary}
\begin{proof}
    Since $F \prec p\cdot K_k$ for some $p$ and $k$, \cref{thm:complete-config} guarantees a constant $C > 0$ so that $\forb(m,r,F) \leq Cm^{k-1}(r-1)^{m}$ for every $m$ and $r$. We may assume by complementation that $F$ contains at least one 0, in which case the configuration that contains every column with no 0's avoids $F$; this implies that $\forb(m,r,F) \geq (r-1)^m$ for every $m, r \in \N$. If $r \geq 3$ is fixed, then the logarithmic growth rates of the lower and upper bounds are asymptotically equal as $m \to \infty$.
\end{proof}

The trivial bound $\forb(m,r,F) \leq r^m$ combined with the lower bound $\forb(m,r,F) \geq (r-1)^m$ shows that $\forb(m,r,F) = \Theta(r^m)$ if $m$ is fixed and $\forb(m,r,F)$ is regarded as a function of $r$.

Going back to exact results, let $K_k^s$ denote the $k\times \binom{k}{s}$ configuration of zeros and ones in which every column contains $s$ ones, called the \textit{complete uniform configuration of weight $s$}. F\"uredi and Quinn proved in \cite{furedi1984traces} that $\forb(m,K_k^s) = \sum_{i=0}^{k-1} \binom{m}{i}$. The configuration where $s$ ones never appear above $k-s$ zeros provides the lower bound; since $K_k^s \prec K_k$, Sauer's theorem provides the upper bound. The construction easily extends, yielding the following result.

\begin{proposition}\label{thm:complete-uniform}
If $s \leq k$, then $\forb(m,r,K_k^s) = \sum_{i=0}^{k-1}\binom{m}{i}(r-1)^{m-i}$. If $(r-2)^{m-k} \geq p-1$, then $\forb(m,r,p\cdot K_k^s) = \forb(m,r,K_k^s) + (p-1)\binom{m}{k}$.
\end{proposition}
\begin{proof}
Let $A_n$ be the $n$-rowed configuration that contains every column in which $s$ ones do not appear above $k-s$ zeros. The sequence $(A_n)$ satisfies the conditions of \cref{thm:01-equal}, and an identical calculation to the one in the proof of \cref{thm:complete-config} proves the first statement.

The proof of the upper bound for the second statement is identical to the one in \cref{thm:complete-config}. For the lower bound, let $A$ be the configuration that contains every column where $s$ ones never appear above $k-s$ zeros; this configuration avoids $K_k^s$. If $(r-2)^{m-k} \geq p-1$, then for each $X \in \binom{[m]}{k}$ we can append $p-1$ columns to $A$ that have $s$ ones above $k-s$ zeros in the rows of $X$ and non-binary digits elsewhere. The resulting configuration avoids $p\cdot K_k^s$ and has $\forb(m,r,p\cdot K_k^s) = \forb(m,r,K_k^s) + (p-1)\binom{m}{k}$ columns.
\end{proof}

A matrix is called \textit{$p$-simple} if each column has multiplicity at most $p$.

\begin{corollary}\label{thm:p-simple-with-p-block}
Assume that $F$ is a $k$-rowed $p$-simple matrix such that $p\cdot K_k^s \prec F$ for some $0 \leq s \leq k$. If $(r-2)^{m-k} \geq p-1$, then
\[ \forb(m,r,F) = \sum_{n=0}^{k-1} \binom{m}{n}(r-1)^{m-n} + (p-1)\binom{m}{k}. \]
If $F$ is simple and $K_k^s \prec F$, then $\forb(m,r,F) = \forb(m,r,K_k^s) = \forb(m,r,K_k)$ for all $m \in \N$ and $r \geq 2$.
\end{corollary}
\begin{proof}
Since $p\cdot K_k^s \prec F \prec p\cdot K_k$, the statement follows from \cref{thm:complete-config,thm:complete-uniform}.
\end{proof}

The result for non-simple matrices in \cref{thm:complete-uniform} is only applicable when $r > 3$. The argument can be modified to show that $\forb(m,3,p\cdot K_k^s)$ is at most an additive constant away from $\sum_{i=0}^{k-1} \binom{m}{i}2^{m-i} + (p-1)\binom{m}{k}$.

\begin{proposition}\label{thm:non-simple-complete-uniform}
Suppose $p > 1$ and $a = \lceil \log_2(p-1)\rceil$. Then
\[ \forb(m,3,p\cdot K_k^s) \geq \sum_{i=0}^{k-1} \binom{m}{i}2^{m-i} + (p-1) \left(\binom{m}{k} - \binom{k+a-1}{k}\right). \]
\end{proposition}
\begin{proof}
Let $A$ be the configuration with all columns that do not contain $s$ ones above $k-s$ zeros. For every $k$-set $X$ with elements $i_1 < i_2 < \dots < i_k$ and $i_s + (m - i_{s+1}) - k \geq \log_2(p-1)$, we may append $p-1$ columns to $A$ with entries $\mathfrak{c}_i$ given by
\[ \begin{cases}
    \mathfrak{c}_i = 1         &\text{if } i \in X \text{ and } i \leq i_s\\
    \mathfrak{c}_i \in \{0,2\} &\text{if } i \notin X \text{ and } i \leq i_s\\
    \mathfrak{c}_i = 2         &\text{if } i_s < i < i_{s+1}\\
    \mathfrak{c}_i = 0         &\text{if } i \in X \text{ and } i \geq i_{s+1}\\
    \mathfrak{c}_i \in \{1,2\} &\text{if } i \notin X \text{ and } i \geq i_{s+1}.
\end{cases} \]
For each such column $\mathfrak{c}$, there is exactly one $k$-set $S$ (namely $S = X$) so that $\mathfrak{c}\vert_S$ is $s$ ones above $k-s$ zeros. Therefore, the resulting configuration $A'$ avoids $p\cdot K_k^s$.

To determine the number of columns added to $A$, we count the number of choices of $X$ with $i_s + (m - i_{s+1}) - k < \log_2(p-1)$. The number of choices with $i_s + (m - i_{s+1}) - k = b$ is $\binom{k-1+b}{k-1}$, so the number of choices of $X$ not covered in our strategy is
\[ \sum_{b=0}^{a-1} \binom{k-1+b}{k-1} = \binom{k+a-1}{k}. \]
In total, then $A'$ contains $(p-1)\big(\binom{m}{k} - \binom{k+a-1}{k}\big)$ more columns than $A$.
\end{proof}

\begin{corollary}\label{thm:complete-uniform-p=2}
$\forb(m,3,2\cdot K_k^s) = \sum_{i=0}^{k-1} \binom{m}{i}(r-1)^{m-i} + \binom{m}{k}$.
\end{corollary}
\begin{proof}
Applying \cref{thm:non-simple-complete-uniform} with $p=2$ gives the lower bound, and the upper bound follows from \cref{thm:complete-config} together with $2\cdot K_k^s \prec 2\cdot K_k$.
\end{proof}

\section{Identity matrices}\label{sec:identity-matrices}

Noting that $I_k = K_k^1$ yields the following corollary of \cref{thm:complete-uniform}.

\begin{corollary}\label{thm:identity-r>3}
If $r > 3$, then $\forb(m,r,p\cdot I_k) = \sum_{i=0}^{k-1}\binom{m}{i}(r-1)^{m-i} + (p-1)\binom{m}{k}$ for all $m$ such that $(r-2)^{m-k} \geq p-1$.
\end{corollary}

The main result of this section is a stability theorem for $I_2$. It would be interesting to see similar stability theorems for other complete uniform configurations.

With each configuration $A \in \Avoid(m,r,I_2)$ we can associate a tournament on $m$ vertices. Direct an edge from $i$ to $j$ if there is no column in which 0 appears in row $i$ and 1 appears in row $j$. If both $ij$ and $ji$ are possible edges, choose just one. Since $A$ avoids $I_2$, there must be an edge between each pair of vertices, so this construction gives a tournament $T_A$ on $m$ vertices.

\begin{proposition}\label{thm:stability-I2}
Let $r \geq 2$ and $A \in \Avoid(m,r,I_2)$ such that $T_A$ is not transitive. Then $\lvert A \rvert \leq m(r-1)^{m-1} + (r-1)^m - 2(r-1)^{m-3}$.
\end{proposition}
\begin{proof}
We first prove the case $r=2$: If $A \in \Avoid(k,2,I_2)$ such that $T_A$ is not transitive, then $\lvert A \rvert \leq m-1$. Since $T_A$ is not transitive, it contains a 3-cycle $ijk$. The only possible columns in $A\vert_{\{i,j,k\}}$ are $\sbb 0 \\ 0 \\ 0 \seb$ and $\sbb 1 \\ 1 \\ 1\seb$. If we delete rows $i$ and $j$, then the resulting configuration $A'$ is simple and avoids $I_2$, so $\lvert A \rvert = \lvert A' \rvert \leq \forb(m-2,I_2) = m-1$.

We now proceed with the general case. Suppose that $A \in \Avoid(m,r,I_2)$ with $T_A$ not transitive. As before, there is a 3-cycle $ijk$ in $T_A$. Applying the argument used in the proof of \cref{thm:upper-01-bound} and splitting the sum over sets that do or do not contain $\{i,j,k\}$ gives the bound
\begin{align*}
    \forb(m,r,I_2)
    &\leq \sum_{\substack{X \subseteq [m] \\ \mathclap{\{i,j,k\} \not\subseteq X}}} (r-2)^{m-\lvert X\rvert}(\lvert X \rvert+1)\,
        + \mspace{-13mu}\sum_{\substack{X \subseteq [m] \\ \mathclap{\{i,j,k\} \subseteq X}}} (r-2)^{m-\lvert X\rvert}(\lvert X \rvert-1)\\
    &= \sum_{X \subseteq [m]} (r-2)^{m-\lvert X\rvert}(\lvert X \rvert+1)\,
        -\, 2\sum_{\substack{X \subseteq [m] \\ \mathclap{\{i,j,k\} \subseteq X}}}  (r-2)^{m-\lvert X\rvert}.
\end{align*}
The left sum simplifies to
\[
    \sum_{k=0}^m \binom{m}{k}(r-2)^{m-k}(k+1) = m(r-1)^{m-1} + (r-1)^m,
\]
and the right sum is
\[
    \sum_{k=3}^m \binom{m-3}{k-3} (r-2)^{m-k}
    = \sum_{k=0}^{m-3} \binom{m-3}{k} (r-2)^{m-3-k}
    = (r-1)^{m-3}.
\]
Combining the two evaluations completes the proof.
\end{proof}

\begin{theorem}\label{thm:I2-unique}
For each integer $r \geq 2$, there is a unique extremal $r$-configuration with $m$ rows that avoids $I_2$.
\end{theorem}
\begin{proof}
By \cref{thm:stability-I2}, if $A$ is extremal, then $T_A$ is transitive. Therefore there is an ordering $i_1,\dots,i_m$ of $[m]$ so that $i_si_t$ is an edge of $T$ if and only if $s < t$. After permuting the rows of $A$ according to this order, no 0 appears above a 1. There are $m(r-1)^{m-1}$ such columns that contain a 0 and $(r-1)^m$ columns with no $0$. Since $A$ is extremal, it contains all these columns. Up to row and column permutation, therefore, $A$ is unique.
\end{proof}

Thus, there is a gap between the unique extremal configuration that avoids $I_2$ and any other configuration that avoids $I_2$ but is not a subconfiguration of the extremal one.

In another direction, \cref{thm:complete-config,thm:complete-uniform} and \cref{thm:complete-uniform-p=2} show that $\forb(m,3,p\cdot I_2) = \forb(m,3,p\cdot K_2)$ when $p=1$ or $p=2$. However, equality does not hold for higher values of $p$. The following exact evaluation of $\forb(m,r,3\cdot I_2)$ shows that $\forb(m,3,p\cdot I_k) \neq \forb(m,3,p\cdot K_k)$ in general. In contrast, \cref{thm:identity-r>3} states that $\forb(m,r,p\cdot I_k) = \forb(m,r,p\cdot K_k)$ for every $p \geq 1$ when $r > 3$.

\begin{proposition}
If $m \geq 4$, then $\forb(m,3,3\cdot I_2)
= \forb(m,3,3\cdot K_2)-1$.
\end{proposition}
\begin{proof}
Let $A$ be the configuration constructed in the proof of \cref{thm:non-simple-complete-uniform} with $\forb(m,3,3\cdot K_2) - 2$ columns that avoids $3\cdot I_2$. Appending the column $\mathfrak{c}$ with $\mathfrak{c}_1 = 1$, $\mathfrak{c}_m = 0$, and $\mathfrak{c}_i = 2$ for every $1 < i < m$ creates a configuration with $\forb(m,3,3\cdot K_2) - 1$ columns that avoids $3\cdot I_2$.

We now show that any 3-configuration that avoids $3\cdot I_2$ has at most $\forb(m,3,3\cdot K_2) - 1$ columns. In each pair of rows, either $\sbb 0 \\1 \seb$ or $\sbb 1 \\ 0\seb$ appears at most twice. Permuting the corresponding columns to the right end of the configuration $A$, we create a decomposition $A = [BC]$ where $\lvert C \rvert \leq 2\binom{m}{2}$ and $B$ avoids $I_2$. If $B$ is not the unique extremal configuration that avoids $I_2$, then
\[
    \lvert A \rvert
    = \lvert B \rvert + \lvert C\rvert
    \leq \forb(m,3,K_2) - 1 + 2\binom{m}{2}
    = \forb(m,3,3\cdot K_2) - 1.
\]
Otherwise, \cref{thm:I2-unique} shows that we may permute the rows of $B$ so it contains every column where no 0 appears above a 1. Since $B$ has at least four rows, $B|_{\{i,j\}}$ contains at least four columns of the form $\left[\begin{smallmatrix} 1 \\ 0\end{smallmatrix}\right]$ for every $i,j \in [m]$ with $i < j$.

We \textit{mark} the pair $i < j$ for each time that 0 appears in row $i$ and 1 appears in row $j$ in the configuration $C$. Since $A$ avoids $3\cdot I_2$ and $B|_{\{i,j\}}$ already contains four columns of the form $\left[\begin{smallmatrix} 1 \\ 0\end{smallmatrix}\right]$, each pair has at most two marks. Each column of $C$ contributes at least one mark. If the pair $(1,m)$ has at most one mark, then there are at most $2\binom{m}{2} - 1$ columns in $C$. If $(1,m)$ has two marks, then there is a column $\mathfrak{c}$ in $C$ with $\mathfrak{c}_1 = 0$, $\mathfrak{c}_m = 1$, and $\mathfrak{c}_s \neq 2$ for some $1 < s < m$. In this case the column $\mathfrak{c}$ contributes at least two marks: one for $(1,m)$, and one for either $(1,s)$ or $(s,m)$. So in this case, too, there are at most $2\binom{m}{2}-1$ columns in $C$. In either case,
\[
    \lvert A \rvert
    = \lvert B\rvert + \lvert C\rvert
    \leq \forb(m,3,3\cdot K_2) - 1,
\]
proving the lower bound.
\end{proof}

The upper bound in this argument shows that $\forb(m,3,p\cdot K_2) < \forb(m,3,p\cdot I_2)$ for all $p \geq 3$. Indeed, by following this mark argument, one can calculate exact forbidden numbers for larger $p$. It's not too hard to show, for example, that $\forb(m,3,4\cdot I_2) = \forb(m,3,4\cdot K_2) - 2$ and $\forb(m,3,5\cdot I_2) = \forb(m,3,5\cdot K_2) - 5$. The computations, however, rapidly become rather case-heavy as $p$ increases. In general, the mark argument can be extended to show that the difference between $\forb(m,3,p\cdot I_2)$ and $\forb(m,3,p\cdot K_2)$ is superlinear in $p$; for example,
\begin{equation}
    \forb(m,3,p\cdot I_2) \leq \forb(m,3,p\cdot K_2) - \frac{1}{4} (p-1)\log_2(p-1)\big(\!\log_2(p-1)-1\big),
\end{equation}
although this is not sharp.

\section{Block matrices}\label{sec:block-matrices}
\begin{proposition}\label{thm:vertical-block}
If $(r-2)^{m-a-b} \geq p-1$ , then
\begin{multline}\label{eqn:vertical-block-matrix}
    \forb\left(m,r,{\sbb \mathbf{0}_{a\times p} \\ \mathbf{1}_{b\times p} \seb} \right) =
    \sum_{\ell=0}^{a-1}\binom{m}{\ell}(r-1)^{m-\ell}
    + \sum_{k=0}^{b-1}\binom{m}{k}(r-1)^{m-k}\\
    -\sum_{\ell=0}^{a-1}\sum_{k=0}^{b-1}\binom{m}{\ell}\binom{m-\ell}{k}(r-2)^{m-\ell-k}
    + (p-1)\binom{m}{a}\binom{m-a}{b}.
\end{multline}
\end{proposition}
\begin{proof}
Any maximal matrix that avoids $F \coloneqq \sbb \mathbf{0}_{a\times p} \\ \mathbf{1}_{b\times p} \seb$ contains all columns that have fewer than $a$ zeros or fewer than $b$ ones. This accounts for the first three terms of \eqref{eqn:vertical-block-matrix}. Thus we need only bound the number of columns that contain at least $a$ zeros and at least $b$ ones. There are $(r-2)^{m-a-b}$ columns that contain exactly $a$ zeros and $b$ ones for a fixed $a$-set $X$ and $b$-set $Y$ of rows. If $(r-2)^{m-a-b} \geq p-1$, then for each disjoint $X, Y \subseteq [m]$ with $\lvert X \rvert = a$ and $\lvert Y \rvert = b$, we may take $p-1$ columns with 0's in the rows in $X$ and 1's in the rows of $Y$ and entries in $\{2,\dots,r-1\}$ elsewhere. This is $(p-1)\binom{m}{a}\binom{m-a}{b}$ columns, which provides the lower bound.

For the upper bound, we again use a mark argument. Consider the set of ordered pairs $(X,Y)$ where $X, Y \subseteq [m]$ are disjoint, $\lvert X \rvert = a$, and $\lvert Y \rvert = b$. Given a matrix $A$, we place a mark on the pair $(X,Y)$ for every column $\mathfrak{c} \in A$ such that $\mathfrak{c}\vert_X$ contains only zeros and $\mathfrak{c}\vert_Y$ contains only ones. There can be at most $(p-1)\binom{m}{a}\binom{m-a}{b}$ marks in total if the matrix $A$ avoids $F$. Every column that contains at least $a$ zeros and $b$ ones contributes at least one mark, so there are at most $(p-1)\binom{m}{a}\binom{m-a}{b}$ such columns, which gives the upper bound.
\end{proof}

\begin{corollary}\label{thm:0-1vector}
If $(r-2)^{m-2} \geq p-1$, then
\[
    \forb(m,r, p \cdot \sbb 0  \\ 1 \seb) = 2(r-1)^m - (r-2)^m + (p-1)m(m-1).
\]
\end{corollary}

\section{Forbidden configurations with 2 rows}\label{sec:two-rows}

We define the general 2-rowed $(0,1)$--forbidden configuration
\begin{equation}
    F(a,b,c,d) =
    \bigg[
    \underbrace{
        \begin{matrix} 0 & \cdots & 0 \\ 0 & \cdots & 0 \end{matrix}
        }_a
    \underbrace{
        \begin{matrix} 1 & \cdots & 1 \\ 0 & \cdots & 0 \end{matrix}
        }_b
    \underbrace{
        \begin{matrix} 0 & \cdots & 0 \\ 1 & \cdots & 1 \end{matrix}
        }_c
    \underbrace{
        \begin{matrix} 1 & \cdots & 1 \\ 1 & \cdots & 1 \end{matrix}
        }_d
    \bigg].
\end{equation}

Our main tools will be two reduction lemmas.

\begin{lemma}[Reduction Lemma for $r > 3$]\label{thm:reduction-r>3}
Suppose $b,c \geq 1$ and set $b' = \min\{b,c\}$. If $(r-2)^{m-2} \geq 2(\max\{b,c\} - 1)$, then
\[ \forb\!\big(m,r,F(a,b,c,d)\big) = \forb\!\big(m,r,F(a,b',b',d)\big). \]
\end{lemma}
\begin{proof}
If $b=c$ the statement is trivial, so suppose without loss of generality that $b < c$. We set $F\coloneqq F(a,b,c,d)$ and $F' = F(a,b,b,d)$. It follows from $F' \prec F$ that $\forb(m,r,F') \leq \forb(m,r,F)$. To prove the reverse inequality, we want to show that $\lvert A\rvert \leq \forb(m,r,F')$ for every $A \in \Avoid(m,r,F)$. This is true if $A$ avoids $F'$, so suppose instead that $F' \prec A$. By permuting the rows of $A$, we may assume that some instance of $F'$ appears in its first two rows. We write $A$ in the block form
\begin{equation}
    A = \begin{bmatrix}
    0 & 0 & 1 & \cdots & r-1\\
    0 & 1 & 0 & \cdots & r-1 \\
    A_{0,0} & A_{0,1} & A_{1,0} & \cdots & A_{r-1,r-1}
    \end{bmatrix}.
\end{equation}

Because $F'$ appears in the first two rows, we know that $\lvert A_{0,0} \rvert \geq a$, that $\lvert A_{0,1}\rvert, \lvert A_{1,0}\rvert \geq b$, and that $\lvert A_{1,1}\rvert \geq d$. If either of $A_{0,1}$ or $A_{1,0}$ contains at least $c$ columns, then $A$ contains $F$ in the first two rows. But $A$ avoids $F$, so $\lvert A_{1,0}\rvert, \lvert A_{0,1}\rvert < c$. We assumed that $(r-2)^{m-2} \geq 2 (c - 1)$, so it is possible to delete the columns with $\sbb 0 \\ 1\seb$ or $\sbb 1 \\ 0 \seb$ in the first two rows and append $\lvert A_{0,1}\rvert + \lvert A_{1,0}\rvert$ distinct columns $\mathfrak{c}$ to $A$ with $\mathfrak{c}_1 = 0$, $\mathfrak{c}_2 = 1$, and $\mathfrak{c}_i \notin\{0,1\}$ for $i > 2$. The resulting configuration does not contain $I_2$ in its first two rows, so it does not contain $F'$ in the first two rows, either. Moreover, this operation does not create a new instance of $F'$ in $A$.

Iterating this process for every appearance of $F'$ in $A$ produces a matrix with the same number of columns as $A$ that avoids $F'$. Thus $\lvert A\rvert \leq \forb(m,r,F')$, as desired.
\end{proof}

\begin{theorem}[Forbidden numbers for $2$-rowed $(0,1)$-matrices with $r > 3$] Let $F = F(a,b,c,d)$ and $\alpha = \max\{a,d, \min\{b,c\}\}$, and suppose $(r-2)^{m-2} \geq 2 \max\{a,b,c,d\}$. If $\alpha > 0$, then
\begin{equation}\label{eqn:2-rowed-r>3}
    \forb\!\big(m,r,F(a,b,c,d)\big) = m(r-1)^{m-1} + (r-1)^m + (\alpha - 1)\binom{m}{2}.
\end{equation}
Otherwise, $F = p\cdot \sbb 0 \\ 1 \seb$ and 
\begin{equation}
\forb(m,r,p \cdot \sbb 0 \\ 1 \seb) = 2(r-1)^m - (r-2)^m + (p-1)m(m-1).
\end{equation}
\end{theorem}
\begin{proof}
The case $F = p \cdot \sbb 0 \\ 1 \seb$ is given by \cref{thm:0-1vector}. We prove the statement for $\alpha > 0$ in cases. 

\textit{Case 1:} $\alpha = a$ and $b,c \geq 1$. By \cref{thm:reduction-r>3}, we may assume that $b=c$. Then $\mathbf{0}_{2\times a} \prec F$ and $F$ is $a$-simple, so the statement follows from \cref{thm:p-simple-with-p-block}. Taking the $(0,1)$-complement of $F$ handles the case $\alpha = d$ with $b,c\geq 1$.

\textit{Case 2:} $\alpha = \min\{b,c\}$. This implies $b,c\geq 1$, so by \cref{thm:reduction-r>3}, we may assume $b=c$. Then $c \cdot I_2 \prec F \prec c\cdot K_2$, so the upper bound follows from \cref{thm:complete-config} and the lower bound from \cref{thm:identity-r>3} with $k=2$.

\textit{Case 3:} $b=0$ or $c=0$. By possibly taking the complement, we may assume that $b=0$. Since the arguments are symmetric, suppose $a \geq d$, which implies that $\alpha = a$. Then $F \prec F(a,1,\max\{1,c\},d)$, so by \cref{thm:reduction-r>3},
\[  \forb(m,r,a\cdot \mathbf{0}_2)
    \leq \forb(m,r,F)
    \leq \forb(m,r,F(a,1,1,d)). \]
The lower and upper bounds are equal by \cref{thm:p-simple-with-p-block} and Case 1.
\end{proof}

Proving a reduction lemma for $r=3$ requires a different approach.

\begin{lemma}[Reduction Lemma for $r=3$]\label{thm:reduction-r=3}
Let $b' = \min\{b,c\}$. If $2^{m-2} \geq (\max\{a,b,c,d\}-1)m^2$ and $b' \geq 1$, then
\[ \forb\!\big(m,3,F(a,b,c,d)\big) = \forb\!\big(m,3,F(a,b',b',d)\big). \]
\end{lemma}
\begin{proof}
Let $p = \max\{a,b,c,d\}$, so that $F\coloneqq F(a,b,c,d)$ is $p$-simple, and set $F' = F(a,b',b',d)$. As above, $\forb(m,r,F') \leq \forb(m,r,F)$ follows from the observation that $F' \prec F$.

Now let $A \in \Avoid(m,3,F)$. We want to show that $\lvert A \rvert \leq \forb(m,r,F')$. If $A$ does not contain $F'$, this is clear, so we assume that $F' \prec A$. We write $A$ in block form as
\begin{equation}
    A = \left[
\begin{matrix}
        0 & 1 \\
        1 & 0 \\
        A_{0,1} & A_{1,0\,}
    \end{matrix}\right\vert \left.
    \begin{matrix}
        0 & 1 & 0 & 1 & 2 & 2 & 2 \\
        0 & 1 & 2 & 2 & 0 & 1 & 2 \\
        B & C & D & E & F & G & H
    \end{matrix}\right].
\end{equation}
By possibly taking the complement of $F$, we may assume that $b \leq c$. Moreover, since the statement is trivial if $b = c$, we assume that strict inequality holds. Since $F' \prec A$ but $F \nprec A$, we have that $b \leq \lvert A_{0,1} \rvert, \lvert A_{1,0} \rvert < c$. If there is a column in $B$ that is not in $D$, then we may delete the column $\sbb 0 \\ 0 \\ v \seb$ and insert the column $\sbb 0 \\ 2 \\ v \seb$ without introducing $F$ as a configuration. By replacing binary digits in the first two rows with 2's in this manner, we may assume that\vspace{-2em}
\[  \begin{minipage}{0.3\textwidth}
        \begin{align*}
            B \subseteq D \subseteq H \\
            B \subseteq F \subseteq H
        \end{align*}
    \end{minipage}
    \raisebox{-0.5em}{and}
    \begin{minipage}{0.3\textwidth}
        \begin{align*}
            &C \subseteq E \subseteq H \\
            &C \subseteq G \subseteq H.
        \end{align*}
    \end{minipage} \]

The matrices $\sbb 0 & 1 & 2 \\ D & E & H \seb$ and $\sbb 0 & 1 & 2 \\ F & G & H \seb$ both avoid $F$, so $\lvert DEH \rvert + \lvert FGH \rvert \leq 2\forb(m-1,3,F)$. Also, $B \cup C \subseteq H$. From inclusion-exclusion, $\lvert B \rvert + \lvert C \rvert - \lvert H \rvert \leq \lvert B \cap C\rvert$. Because $A$ avoids $p\cdot K_2$, we know that $B \cap C$ avoids $p\cdot K_1$, so $\lvert B \cap C \rvert \leq 2^{m-2} + (p-1)(m-2)$. Therefore
\[ \lvert A \rvert \leq 
    2(c-1) + 2\forb(m-1,3,F) + 2^{m-2} + (p-1)(m-2). \]
If $2^{m-2} \geq (p-1)m^2$, then \cref{thm:complete-config} implies
\begin{align*}
    \lvert A \rvert
    &\leq 2(p-1) + 2 \forb(m-1,3,p\cdot K_2) + 2^{m-2} + (p-1)(m-2)\\
    &= (m-1)2^{m-1} + 2^m + 2^{m-2} + (p-1)m(m-1) + (p-1)m\\
    &\leq m2^{m-1} + 2^m
    = \forb(m,3,I_2).
\end{align*}
Since $I_2 \prec F'$, this shows that $\lvert A \rvert \leq \forb(m,3,F')$, completing the proof.
\end{proof}

\begin{theorem}
Suppose $\max\{a,d\} \geq \min\{b,c\}$. If $2^{m-2}\geq (\max\{a,b,c,d\}-1)m^2$, then
\begin{equation}
\forb\!\big(m,3,F(a,b,c,d)\big) = (r-1)^m + m(r-1)^{m-1} + (\max\{a,d\}-1)\binom{m}{2}.
\end{equation}
\end{theorem}
\begin{proof}
Since all arguments are symmetric, we assume that $a\leq d$ and $b \leq c$. If $b \geq 1$, then
\[ d\cdot \mathbf{1}_2 \prec F(a,b,b,d) \prec d\cdot K_2. \]
The forbidden numbers of both bounding configurations are equal by \cref{thm:p-simple-with-p-block}, and applying \cref{thm:reduction-r=3} shows that $F(a,b,b,d)$ and $F(a,b,c,d)$ have the same forbidden number. If $b = 0$ then $F(a,0,c,d) \prec F(a,1,\max\{1,c\},d)$. We have $d \cdot \mathbf{1}_2 \prec F(a,0,c,d)$, and by \cref{thm:reduction-r=3},
\[ \forb\!\big(m,3,F(a,1,\max\{1,c\},d)\big)
    = \forb\!\big(m,3,F(a,1,1,d)\big)
    \leq \forb(m,3,d \cdot K_2). \]
Again the upper and lower bounds are equal by \cref{thm:p-simple-with-p-block}.
\end{proof}

The remaining question is to evaluate $\forb\!\big(m,3,F(a,b,c,d)\big)$ when $\min\{b,c\} > \max\{a,d\}$. By the Reduction Lemma, we need only consider the case $b=c$. The smallest 2-rowed $(0,1)$-matrix whose forbidden number is not known when $r=3$ is 
\[  \begin{bmatrix}
    0 & 1 & 1 & 0 & 0 \\
    0 & 0 & 0 & 1 & 1
    \end{bmatrix}. \]

\begin{section}{Forbidden configurations with 3 rows}\label{sec:three-rows}

\cref{thm:01-equal} provides a handful of results on 3-rowed forbidden matrices for free.

\begin{corollary}\label{thm:3-row}
The following forbidden numbers are exact for all $r\geq 2$ when $m\geq 3$.
\[
\begin{array}{c | c }
    F & \forb(m,r,F)\\\hline
    \rule[-1em]{0pt}{2.5em}\sbb 0 \\ 1 \\ 1 \seb \text{ or } \sbb 1 & 1 \\ 0 & 1 \\ 0 & 0 \seb  & m(r-1)^{m-1} + 2(r-1)^m - (r-2)^m - m(r-2)^{m-1}\\\hline
    \rule{0pt}{1.5em}\sbb 1 & 0 \\ 0 & 1 \\ 0 & 0 \seb \text{ or } \sbb 1 & 0 & 1 & 0 \\ 0 & 1 & 1 & 1 \\ 0 & 0 & 0 & 1 \seb  & 2m(r-1)^{m-1} + (r-2)^m.
\end{array}
\]
\end{corollary}
\begin{proof}
Theorem 3.2 of \cite{small-forbidden} proves that the configuration $A_m = [\mathbf{0}_m I_m \mathbf{1}_m]$ is extremal for the second matrix in the first row of the table when $r=2$, and it is not too hard to see that $A_m$ is extremal for the first matrix, as well. The sequence $(A_m)$ satisfies the conditions of \cref{thm:01-equal}, and $\lvert A_m \rvert= m+1$ if $m\in \{0,1\}$ and $\lvert A_m\rvert = m+2$ otherwise, so for either matrix $F$, we have
\[
    \forb(m,r,F) = \sum_{k=0}^m \binom{m}{k}(r-2)^{m-k}\forb(k,F);
\]
simplifying this sum yields the expression in the first row of the table.

Let $U_m$ denote the $m\times m$ upper-triangular matrix with 1's on and above the diagonal and 0's elsewhere. Theorem 3.3 of \cite{small-forbidden} proves that $A_m = U_m \cup I_m^c \cup \mathbf{0}_m$ is an extremal configuration (when $r=2$) for both matrices in the second row of the table. Both $U_m$ and $I_m^c$ have a column with exactly one 0 and are otherwise disjoint, so $\lvert A_m\rvert = 2m$ for $m\geq 1$, and $\lvert A_0 \rvert = 1$ by convention. As before, $(A_m)$ satisfies the conditions of \cref{thm:01-equal}, so for either forbidden matrix $F$,
\begin{align*}
    \forb(m,r,F)
    &= \sum_{k=0}^m \binom{m}{k}(r-2)^{m-k}\forb(k,F)\\
    &= \sum_{k=0}^m \binom{m}{k}(r-2)^{m-k}2k + (r-2)^m\\
    &= 2m(r-1)^{m-1} + (r-2)^m,
\end{align*}
as claimed.
\end{proof}

The matrices $\sbb 1 & 0 & 1 \\ 0 & 1  & 1\\ 0 & 0  & 0\seb$ and $\sbb 1 & 0 & 0 \\ 0 & 1  & 1\\ 0 & 0  & 1\seb$ are sandwiched between the two matrices in the second row of \cref{thm:3-row}, so they have the same forbidden number. Our results, together with $(0,1)$-complementation, evaluate the exact forbidden number for all $3\times 2$ and $3\times 3$ simple matrices for $r\geq 3$ except $F = \sbb 1 & 0 \\ 0 & 1 \\ 0 & 1 \seb$. We can, however, bound $\forb(m,r,F)$ using our previous results. Theorem 3.3 in \cite{anstee-preprint} states that $\forb(m,F) \leq \frac{3}{2}m + 1$; applications of our \cref{thm:upper-01-bound} and \cref{thm:complete-uniform} show that
\[ m(r-1)^{m-1} + (r-1)^m \,\leq\, \forb(m,r,F) \,\leq\, \frac{3}{2}m(r-1)^{m-1} + (r-1)^m. \]
\end{section}

\bibliography{bibliography}
\bibliographystyle{amsplain}

\end{document}